\documentclass[
12pt,reqno]{amsart}
\usepackage{verbatim}
\usepackage{amssymb}

\usepackage{enumerate}
\usepackage[active]{srcltx}
\numberwithin{equation}{section}

\usepackage{t1enc}
\usepackage[utf8x]{inputenc}

\newtheorem{theorem}{Theorem}[section]
\newtheorem{lemma}[theorem]{Lemma}
\newtheorem{corollary}[theorem]{Corollary}

\newtheorem{problem}[theorem]{Problem}

\theoremstyle{definition}
\newtheorem{definition}[theorem]{Definition}

\theoremstyle{remark}

\newcommand{\itemprefix}{}
\makeatletter
\newcommand{\myitem}{%
\item\protected@edef\@currentlabel{\itemprefix\theenumi}%
}
\makeatother

\newcommand{\mc}[1]{\mathcal{#1}}

\newcommand{\mf}[1]{\mathfrak{#1}}

\author[I. Juh\'asz]{Istv\'an Juh\'asz}
\address      {HUN-REN Alfr\'ed Rényi Institute of Mathematics%
}
\email{juhasz@renyi.hu}

\author[J. van Mill]{Jan van Mill}
\address{University of Amsterdam}
\email{j.vanMill@uva.nl}

\author[L. Soukup]{Lajos Soukup}
\address
      {HUN-REN Alfr{\'e}d R{\'e}nyi Institute of Mathematics}
\email{soukup@renyi.hu}

\author[Z. Szentmikl\'ossy]{Zolt\'an Szentmikl\'ossy}
\address{Eötvös University of Budapest}
\email{szentmiklossyz@gmail.com}

\makeatletter
  \@namedef{subjclassname@2020}{%
  \textup{2020} Mathematics Subject Classification}
   \makeatother

\subjclass[2020]{54A25, 54A35, 54D30, 54E52, 54F65, 54G12}
\keywords{$\Delta$-space, compact space, countably compact space, Baire space}

\title{Some new results on $\Delta$-spaces}

\begin{document}

\begin{abstract}
A topological space $X$ is a $\Delta$-space (or $X \in \Delta$) if for any decreasing sequence
$\{A_n : n < \omega\}$ of subsets of $X$ with empty intersection there is a (decreasing)
sequence $\{U_n : n < \omega\}$ of open sets with empty intersection such that $A_n \subset U_n$
for all $n < \omega$. In this note we prove the following results concerning $\Delta$-spaces.

\begin{itemize}
  \item Every $T_3$ countably compact $\Delta$-space is compact.
  \item If there is a $T_1$ crowded Baire  $\Delta$-space then there is an inner model with a measurable cardinal.
  \item If $X \in \Delta$ and $cf \big(o(X) \big) > \omega$ then $|X| < o(X)$. (Here $o(X)$
is the number of open subsets of $X$.)
\end{itemize}

The first two of these provide full and/or partial solutions to problems raised in the literature,
while the third improves a known result.
\end{abstract}

\maketitle

\section{Introduction}

There has recently been quite an interest in the study of $\Delta$-spaces, i.e. spaces $X$
such that for any decreasing sequence
$\{A_n : n < \omega\}$ of subsets of $X$ with empty intersection there is a (decreasing)
sequence $\{U_n : n < \omega\}$ of open sets with empty intersection such that $A_n \subset U_n$
for all $n < \omega$.
Also, the class of these spaces is denoted by $\Delta$.

This interest is perhaps mainly motivated by the fact that for a Tykhonov  space $X$
the locally convex space $C_p(X)$  is distinguished if and only if $X$ is a $\Delta$-space,
see e.g. \cite{KL0}.

On the other hand the simple and attractive definition of this class,
moreover, the large number of number of other interesting features of
the class of $\Delta$-spaces justifies their study.
The aim of our present paper is to add a few new results to this study.

We emphasize that in this work we do not assume that the $\Delta$-spaces in question,
unless otherwise stated, satisfy any separation axiom.

The notation and terminology we use in this paper are pretty standard, so we do not think
it is necessary to waste any space to them.

\section{Countably compact $\Delta$-spaces}

As is mentioned in \cite{KL}, the Proper Forcing Axiom (PFA), or rather  a very deep result in \cite{BDFN}
derived from PFA, implies that every $T_3$ and
countably compact $\,\Delta$-space is actually compact.
The aim of this section is to
show that, in fact, this is already provable in ZFC.

We start with a preliminary definition.

\begin{definition}
(i) If $X$ is any {\em non-compact} topological space, we denote by $\mc C(X)$
the family of all non-compact closed (i.e. closed unbounded) subsets of $X$.\\
(ii) We say that $S \subset X$ is {\em stationary in $X$} iff $S \cap C \ne \emptyset$
for all $C \in \mc C(X)$. We shall denote by $St(X)$ the family of all stationary
subsets of $X$.
\end{definition}

Now, the following simple lemma will play a key role in our proof.

\begin{lemma}\label{lm:sigmacpt}
Assume that $X$ is a non-compact $\Delta$-space and there is a decreasing
sequence $\{S_n : n < \omega\} \subset St(X)$ with empty intersection.
Then $X$ is $\sigma$-compact.
\end{lemma}

\begin{proof}
Indeed, as $X \in \Delta$, there are open sets $\{U_n : n < \omega\}$ such that
$S_n \subset U_n$ for each $n < \omega$ and $\bigcap \{U_n : n < \omega\} = \emptyset$.
But then, as $S_n$ stationary in $X$, the closed set $K_n = X \setminus U_n$ must be compact
for every $n < \omega$, hence $X = \bigcup \{K_n : n < \omega\}$ is indeed $\sigma$-compact.
\end{proof}

Note that if $\{T_n : n < \omega\} \subset St(X)$ are pairwise disjoint then putting
$S_n = \bigcup \{T_m :  n \le m < \omega\}$ the
sequence $\{S_n : n < \omega\} \subset St(X)$ is decreasing with empty intersection,
hence we may apply Lemma \ref{lm:sigmacpt} to conclude that $X$ is $\sigma$-compact.

We are now ready to formulate and prove the main result of this section.

\begin{theorem}\label{tm:CC=C}
Every $T_3$ and countably compact $\,\Delta$-space is compact.
\end{theorem}

\begin{proof}
Our proof of this is indirect: Assume, on the contrary, that
$$\mc E = \{X \in \Delta \cap T_3 : X \text{ is countably compact and non-compact} \} \ne \emptyset.$$
We shall show that this assumption leads to a contradiction.
Our first step towards this contradiction is the following claim.

Claim 1. $\mc E \ne \emptyset$ implies that there is $X \in \mc E$ which is {\em locally compact}.

To see this, we first note that every $T_3$ and countably compact $\Delta$-space is scattered.
In fact,  crowded countably compact $T_3$-spaces are $\omega$-resolvable by Theorem 6.9 of \cite{CG},
moreover Proposition 6.6 of \cite{LP} says that no $\Delta$-space can be both Baire and $\omega$-resolvable.
We also refer about this to the first paragraph of the next section.

We shall actually show that if $X \in \mc E$ is of minimum possible scattered height $ht(X) = \alpha$ then
$X$ is locally compact.
To see this, we first show that $\alpha$ has to be a limit ordinal. Indeed, if we had $\alpha = \beta + 1$
then the non-empty top level $X_\beta$ of $X$ has to be finite because $X$ is countably compact.
So, if $\mc U$ is any open cover of $X$ then there is a finite $\mc V \subset \mc U$
that covers $X_\beta$. But then $Y = X \setminus \cup \mc V$  is also
a $T_3$ and countably compact $\Delta$-space with $ht(Y) \le \beta < \alpha$,
hence $Y$ is compact by the minimality of $\alpha$. So, $Y$ is also covered by
finitely many members of $\mc U$, showing that $X$ is also compact, a contradiction.

So, we know that $ht(X) = \alpha$ is a limit ordinal. But then, using that $X$ is $T_3$,
every point $x \in X$ has a closed, hence countably compact, neighbourhood $W_x$ such that
$ht(W_x) = ht(x,X) < \alpha$. But the minimality of $\alpha$ then implies that
$W_x$ is actually compact, i.e. $X$ is indeed locally compact.

Now consider the Alexandrov one-point compactification $A(X) = X \cup \{p\}$ of $X$.
Since any one-point extension of a $\Delta$-space is trivially also a $\Delta$-space, we have $A(X) \in \Delta$.
So, $A(X)$ is a compact $\Delta$-space, hence by Theorem 4.2 of \cite{KL} it has countable tightness.
Consequently, there is a countable subset $H$ of $X$ such that $p \in \overline{H}$,
the closure taken in $A(X)$. Hence $X \cap \overline{H}$ is a locally compact
and separable member of $\mc E$.
This argument clearly also yields the following.

Claim 2. Every locally compact $X \in \mc E$ includes a separable locally compact
member of $\mc E$.

As is well-known and easy to prove, see e.g. \cite{JSSS},
separable scattered $T_3$-spaces have cardinality at most $\mathfrak{c}$, the size
of the continuum. So, it is immediate from the above that from $\mc E \ne \emptyset$
it follows that
$$\mc E_0 = \{X \in \mc E : X \text{ is locally compact and } |X| \le \mf c\} \ne \emptyset$$
as well. Note that for every $X \in \mc E_0$ we obviously have $\mc C(X) \subset \mc E_0$.

Now we distinguish two cases.

Case 1. There is $X \in \mc E_0$ such that for any two $C_0, C_1 \in \mc C(X)$
we have $C_0 \cap C_1 \ne \emptyset$.

Note that in this case we even have $C_0 \cap C_1 \in \mc C(X)$ as well.
Indeed, if $K = C_0 \cap C_1$ would be compact then, as $X$ is locally compact,
we could find an open neighborhood $U$ of $K$ such that its closure in $X$
is compact, and then $C_0 \setminus U$ and $C_1 \setminus U$ would be
two disjoint members of $\mc C(X)$, contradicting that we are in Case 1.

We also claim that $\cap \mc C(X) = \emptyset$. This is obvious because
every point $x \in X$ has an open neighborhood $U_x$ with compact closure,
hence $x \notin X \setminus U_x \in \mc C(X)$.

Thus, $\mc C(X)$  is a collection closed under finite intersections and with
empty intersection, hence any ultrafilter $u$ on $X$ extending $\mc C(X)$  is free.
But $|X| \le \mf c$ is less than the first measurable cardinal, if it exists,
hence the ultrafilter $u$ is not $\sigma$-complete. Thus there is a decreasing sequence
$\{S_n : n < \omega\} \subset u$ with empty intersection.
Now, $\mc C(X) \subset u$ clearly implies that every member of $u$ is
stationary in $X$, hence by Lemma \ref{lm:sigmacpt} we have that $X$
is $\sigma$-compact. But countably compact and $\sigma$-compact spaces
are compact, contradicting that $X$ is not compact.

Thus we are left with the following.

Case 2. For every $X \in \mc E_0$ there are two $C_0, C_1 \in \mc C(X)$
such that $C_0 \cap C_1 = \emptyset$.

Claim 3. In this case we have $|X| = \mf c$ for all $X \in \mc E_0$.

Indeed, let us consider any $X \in \mc E_0$. Clearly,
we only have to show that $|X| \ge \mf c$.
To see this, we use the fact that we are in Case 2. and a standard procedure to define a family of sets
$$\{C_s : s \in 2^{< \omega}\} \subset \mc C(X) \subset \mc E_0$$ such that for every
$s \in 2^{< \omega}$ we have
$$C_{s^{\curvearrowright}0} \cup C_{s^{\curvearrowright}1} \subset C_s \text{ and }
C_{s^{\curvearrowright}0} \cap C_{s^{\curvearrowright}1} = \emptyset.$$
Then for any $\omega$-sequence $t \in 2^\omega$, we have the decreasing $\omega$-sequence
$\{ C_{t \upharpoonright n} : n \in \omega\}$ of closed sets in $X$ whose intersection
$C_t \ne \emptyset$ because $X$ is countably compact. So, the family
$\{C_t : t \in 2^\omega\}$ clearly consists of pairwise disjoint non-empty sets,
hence we indeed get $|X| = \mf c$.

For every $X \in \mc E_0$ let us now denote by $\mc C_\omega(X)$ the family of
separable members of $\mc C(X)$. Clearly, then $|\mc C_\omega(X)| \le \mf c$, moreover
every member of $\mc C_\omega(X)$ has cardinality $\mf c$. We can thus apply Kuratowski's
disjoint refinement lemma to find for every $C \in \mc C_\omega(X)$ a $\mf c$-sized
subset $A_C \subset C$ in such a way that the family $\{A_C : C \in \mc C_\omega(X)\}$
is disjoint.

Note that it follows from Claim 2. that if $S \subset X$ intersects every $A_C$
then $S$ is stationary in $X$. But we clearly can find infinitely many, even $\mf c$-many
disjoint such sets. Using Lemma \ref{lm:sigmacpt}, we may thus conclude that $X$
is $\sigma$-compact and hence compact. So, we ran into a contradiction in this case as well,
completing our proof.
\end{proof}

In \cite{LP} the authors write the following: "Whether any countably compact $\Delta$-space has countable tightness
in ZFC is an open problem." Since compact $\Delta$-space do have countable tightness, Theorem \ref{tm:CC=C}
yields an immediate affirmative answer to this question as well.

\section{Baire $\Delta$-spaces}

It is obvious that if $X \in \Delta$ is Baire then the intersection of any decreasing sequence
$\{S_n : n < \omega\}$ of {\em dense} subsets of $X$  is non-empty.
In particular, this clearly implies that $X$ is not $\omega$-resolvable.

Note that
if $X$ has an isolated point then, trivially, the intersection of any family
of dense subsets of $X$ is non-empty. This lead the authors to raise the
following natural question in \cite{KLT}, Problem 4.1: If $X \in \Delta$ is Baire, must $X$
have an isolated point? In this section we give a partial affirmative answer to
this question by showing that the existence of a counterexample implies that there is an inner
model with a measurable cardinal. By the above observation that any Baire space $X \in \Delta$ is not
$\omega$-resolvable, this will be an immediate consequence of our next result.

\begin{theorem}\label{tm:Baire}
If the crowded Baire $T_1$-space $X$ is not $\omega$-resolvable then $X$ has a crowded Baire
irresolvable subspace.
\end{theorem}

\begin{proof}
Illanes proved in \cite{I} that if a space $X$ is not $\omega$-resolvable then there is $0 < n < \omega$
such that $X$ is $n$-resolvable but not $(n + 1)$-resolvable. Thus we may fix a partition
$\{S_i : i < n\}$ of $X$ into $n$ irresolvable dense subsets.

Let $\mc V$ be a maximal disjoint family of $(n + 1)$-resolvable open subsets of $X$.
Then $W = \bigcup \mc V$ cannot be dense in $X$ because it is not $(n + 1)$-resolvable.
Thus we have $U = X \setminus \overline{W} \ne \emptyset$.

We claim that then $U \cap S_i$ is {\em open hereditary irresolvable} (in short OHI) for every $i < n$.
Indeed, if we had a non-empty open $V \subset U$ and $i < n$ such that $V \cap S_i$ is resolvable,
then clearly $V$ would be $(n + 1)$-resolvable, contradicting the maximal disjointness of $\mc V$.
Note also that $V \cap S_i$ is crowded for any open set $V$ because $X$ is $T_1$.
We claim that there are $V \subset U$ non-empty open and $i < n$
for which the crowded and irresolvable subspace $V \cap S_i$ is also Baire.

Indeed, otherwise we could define by induction a decreasing sequence of non-empty
open sets $\{U_i : i < n\}$ as follows. Since $U \cap S_0$ is not Baire,
there is a non-empty
open $U_0 \subset U$ such that $U_0 \cap S_0$ is of first category. If $U_i$
has been defined for some $i < n-1$ so that $U_i \cap S_i$ is of first category then,
by our indirect assumption, we may find  a non-empty
open $U_{i + 1} \subset U_i$ such that $U_{i + 1} \cap S_{i + 1}$ is of first category.
But then at the end we have $$U_{n - 1} = \bigcup \{U_{n - 1} \cap S_i : i < n\},$$
hence the non-empty open  $U_{n - 1} \subset X$, as a finite union of first category sets, would be of
first category in $X$, contradicting that $X$ is Baire.
\end{proof}

Now, as it had been established in \cite{KST} and \cite{KT}, the existence of a crowded Baire
irresolvable space implies that there is an inner model $M$ of the set-theoretic
universe $V$ that contains a measurable cardinal. In other words, we have the
following result that yields a partial solution of problems 4.1 and 4.2 from \cite{KLT}.

\begin{corollary}\label{co:Baire}
If there is no inner model of $V$ containing a measurable cardinal, for instance
if Gödel's axiom of constructibility $V = L$ holds, then every $T_1$  crowded Baire
$\Delta$-space  has an isolated point; consequently, any $T_1$ hereditary Baire
$\Delta$-space is scattered.
\end{corollary}

As is well-known, pseudocompact spaces are Baire, hence we may also ask: is there
a crowded pseudocompact $\Delta$-space? Of course, in view of Theorem \ref{tm:Baire},
the existence of such a space implies the consistency of having a measurable cardinal.
On the other hand, as far as we know, it is open if ZFC implies that all crowded pseudocompact
spaces are $\omega$-resolvable. we actually conjecture that this is true, and hence
there are no crowded and Baire pseudocompact $\Delta$-spaces.

In any case, it follows from the results in \cite{JSSS1} that
the existence of a non-$\omega$-resolvable crowded pseudocompact space would imply
much stronger large cardinals than measurable.

\section{Connection with $o(X)$}

We start with a general, somewhat technical result that gives a condition for $X \notin \Delta$.

\begin{theorem}\label{tm:cf}
Let $X$ be a topological space with $|X| = \kappa$, where $cf(\kappa) > \omega$.
Assume also that there is a family $\mc A \subset [X]^\kappa$ such that $|\mc A| \le \kappa$
and for for every closed set $F \in [X]^\kappa$ there is $A \in \mc A$ with $A \subset F$.
Then $X \notin \Delta$.
\end{theorem}

\begin{proof}
By our assumptions we can apply Kuratowski's
disjoint refinement lemma to $\mc A$ to obtain a disjoint collection
$$\{B_A : A \in \mc A\} \subset [X]^\kappa$$ such that $B_A \subset A$ for all $A \in \mc A$.
Now, it is easy to fix $\omega$-many disjoint sets $\{S_n : n < \omega\}$ such that
$S_n \cap B_A \ne \emptyset$ for any $n < \omega$ and $A \in \mc A$.

It is clear from our assumptions that then we also have $S_n \cap F \ne \emptyset$ for every $n < \omega$ and
for every closed set $F \in [X]^\kappa$.
In other words, then for any open set $U \supset S_n$ we have $|X \setminus U| < \kappa$.
Consequently, if we take open sets $U_n \supset S_n$ for al $n < \omega$,
then $cf(\kappa) > \omega$ implies $$|X \setminus \bigcap \{U_n : n < \omega\}| < \kappa$$
as well. But this clearly implies $X \notin \Delta$.
\end{proof}

Theorem \ref{tm:cf} trivially applies if
we have
$$|\{F \in [X]^\kappa : F \text{ is closed}\}| \le \kappa.$$
In particular, we trivially have this if $o(X) \le \kappa$, where $o(X)$ denotes the number of all
open, or equivalently closed, subsets of $X$.
Thus we have the following
corollary of Theorem \ref{tm:cf}.

\begin{corollary}\label{co:cfo}
If $X \in \Delta$ and $cf \big(o(X) \big) > \omega$ then $|X| < o(X)$.
\end{corollary}

Since $\kappa^\omega = \kappa$ implies $cf(\kappa) > \omega$,
Corollary \ref{co:cfo} improves Theorem 6.1 of \cite{LP} which says that $X \in \Delta$
implies $|X| < o(X)^\omega$.

Clearly, if $X$ is hereditary Lindelöf of weight $\le \mf c$ or if
$|X| = \mf c$ and $X$ is hereditary separable then $o(X) \le \mf c$.
So, if either $X$ is hereditary Lindelöf of weight $\le \mf c$ or hereditary separable
then $X \in \Delta$ implies $|X| < \mf c$.

Thus we are lead to the following problem.

\begin{problem}
Does every ($T_2$ or $T_3$) hereditary Lindelöf $\Delta$-space have cardinality $< \mf c$?
\end{problem}

Provided that the answer to this is affirmative, we may ask the following.

\begin{problem}
Does every ($T_2$ or $T_3$) $\Delta$-space of countable spread have cardinality $< \mf c$?
\end{problem}

\end{document}